\documentclass[a4paper,11pt,reqno]{amsart}
\usepackage{amsmath,amsthm,amssymb}
\usepackage{relsize}
\usepackage[noadjust]{cite}
\usepackage{color}
\usepackage{graphicx}
\usepackage[dvipsnames]{xcolor}
\usepackage{multirow}
\usepackage{booktabs}
\usepackage{verbatim} 

\usepackage{enumitem}
\usepackage{tikz}
\usetikzlibrary{decorations.pathreplacing}

\usepackage[colorlinks, urlcolor=blue, linkcolor=Purple, citecolor=red]{hyperref}
\usepackage[capitalize]{cleveref}

\usepackage{setspace}
\usepackage{enumitem}
\setitemize{itemsep=-1pt}
\setenumerate{itemsep=-1pt}

\usepackage[margin=2.4cm]{geometry}

\usepackage{stmaryrd}
\usepackage{longtable}

\usepackage{array}
\newcolumntype{x}[1]{>{\centering\arraybackslash\hspace{0pt}}p{#1}}

\theoremstyle{definition}
\newtheorem{theorem}{Theorem}[section]
\newtheorem{definition}[theorem]{{{Definition}}}
\newtheorem{example}[theorem]{{{Example}}}

\newtheorem{remark}[theorem]{{{Remark}}}

\newtheorem{corollary}[theorem]{{{Corollary}}}
\newtheorem{proposition}[theorem]{{{Proposition}}}
\newtheorem{lemma}[theorem]{{{Lemma}}}

\newcommand{\numberset}{\mathbb}

\newcommand{\F}{\numberset{F}}
\newcommand{\Fq}{\F_q}

\newcommand{\mP}{\mathcal{P}}

\newcommand\qbin[3]{\left[\begin{matrix} #1 \\ #2 \end{matrix} \right]_{#3}}

\newcommand{\mS}{\mathcal{S}}

\newcommand{\dd}{\mathrm{d}}
\newcommand{\dS}{\dd_{\mathrm{S}}}

\DeclareMathOperator{\Gr}{Gr}

\newcommand{\bF}{\mathbb{F}}

\newcommand{\cA}{\mathcal{A}}
\newcommand{\cB}{\mathcal{B}}

\newcommand{\cK}{\mathcal{K}}
\newcommand{\cP}{\mathcal{P}}

\newcommand{\cS}{\mathcal{S}}

\newcommand{\cU}{\mathcal{U}}
\newcommand{\cW}{\mathcal{W}}

\title{Combinatorial constructions of Schubert subspace codes}

\usepackage[foot]{amsaddr}
\author{}
\author{Gianira N. Alfarano$^1$}
\author{Alessandro Neri$^2$}
\author{Beatrice Toesca$^3$}
\address{$^1$Universit\'e de Rennes, IRMAR, Campus de Beaulieu, F-35042 Rennes Cedex, France.}
\address{$^2$Department of Mathematics and Applications “R. Caccioppoli”, University of Naples Federico II, Via Cintia,
Monte S. Angelo, 80126 Napoli, Italy.}
\address{$^3$Institute of Mathematics, University of Zurich, Switzerland.}

\email{gianira-nicoletta.alfarano@univ-rennes.fr}
\email{alessandro.neri@unina.it}
\email{beatrice.toesca@math.uzh.ch}

\begin{document}
	
\begin{abstract}
We study Schubert subspace codes, which are constant-dimension subspace codes with prescribed intersection conditions with a fixed subspace. Our goal is to construct codes of maximum possible size in the extremal distance cases where a
natural counting upper bound applies. We give two families of constructions. The first one uses a direct-sum decomposition of the ambient space, together with partial spreads and colorings of powers of $q$-Johnson graphs. For this construction, we also prove necessary conditions, which
show how chromatic and clique obstructions arise. The second family is obtained by field
reduction from evasive and scattered subspaces over extension fields. This gives codes whose size can be computed exactly in the scattered case and recovers the only previously known construction as a special case.
\end{abstract}
	
	\maketitle



\section{Introduction}

Subspace codes were introduced in the context of random network coding as a natural error-correcting framework for noncoherent communication over networks. In this model, intermediate nodes are allowed to perform linear combinations of incoming messages before forwarding them. It is natural in the area of network coding to consider the elements of the Grassmannian $\Gr_q(k,n)$  as codewords of a so-called \emph{subspace code}.
The foundational work of K\"otter and Kschischang initiated the systematic study of such codes
and introduced several basic bounds and constructions \cite{koetter2008coding}. Since then, subspace codes, and in particular \emph{constant-dimension codes}, have been studied extensively; see, for instance, \cite{horlemann2018constructions, kurz2021constructions} and references therein. Constant-dimension codes are subsets of the Grassmannian $\Gr_q(k,n)$, and the subspace distance between two codewords $U,V\in \Gr_q(k,n)$ is given by $d_S(U,V)=2(k-\dim(U\cap V))$.
Thus, constructing large constant-dimension codes with a prescribed minimum distance is equivalent to constructing large families of $k$-dimensional subspaces with controlled pairwise intersections.

Several important constructions of constant-dimension codes are based on rank-metric codes.
The lifting of rank-metric codes, and in particular of maximum rank-distance codes, produces large families of subspaces with good distance properties \cite{koetter2008coding}.
This point of view also connects the theory of subspace codes with the geometry of Grassmannians, where subspaces are represented by row spaces of matrices in reduced row echelon form.
In this direction, Ferrers diagram rank-metric codes and multilevel constructions provide flexible methods for building large codes; see \cite{etzion2009error}.
These constructions show that imposing geometric or combinatorial restrictions on the support of codewords can still lead to large and structured families of subspaces.

In this paper, we study \emph{Schubert subspace codes}, which were recently introduced in  \cite{alfarano2024schubert}.
These are subspace codes with an additional geometric constraint: the codewords are required to lie in a fixed Schubert variety of the Grassmannian.
Schubert varieties are among the most classical subvarieties of Grassmannians and they are defined by imposing lower bounds on the dimensions of the intersections with the subspaces in a fixed flag.
Classical references for Schubert varieties over arbitrary fields include \cite{hodge1947methods, fulton1997young}.
In the coding-theoretic setting, Schubert varieties over finite fields were first considered in \cite{ghorpade2000higher}, where the authors introduced the concept of \textit{Schubert codes} as algebraic-geometry codes obtained by evaluating polynomials on points lying on a Schubert variety.
The \emph{Schubert subspace codes} we study in this paper, following \cite{alfarano2024schubert}, are a different family of codes, which are a special class of subspace codes instead of classical linear codes.
In this setting, the Schubert condition amounts to requiring every codeword to satisfy prescribed incidence conditions with fixed subspaces of the ambient space, transforming the problem into an incidence geometry question.
This gives a natural way to study constant-dimension codes under geometric constraints.

In \cite{alfarano2024schubert}, the
authors focused in particular on Schubert varieties defined by an incidence condition with a fixed subspace and constructed optimal examples in the extremal distance case using scattered subspaces.
More precisely, if $U\subseteq \Fq^n$ is a fixed subspace, one considers families of $k$-dimensional subspaces whose intersection with $U$ has dimension at least $1$, while any two distinct codewords intersect trivially.

We continue this line of investigation in this paper, but with a focus on  \emph{$(\ell,t)$-intersecting sets}.
These are precisely Schubert subspace codes defined by the Schubert variety of spaces intersecting a fixed $U$ in dimension at least $\ell$, and with minimum subspace distance at least $2k-2t$.
A basic counting argument gives a natural upper bound for these codes whenever $t\leq \ell-1$, which was already proved in \cite{alfarano2024schubert}.
The main problem addressed in this paper is to determine when this bound is attained and to construct optimal codes.

We focus mainly on the extremal case $t=\ell-1$, which is the largest value of $t$ for which the bound still applies.
Our first construction uses a direct-sum decomposition $\Fq^n=U\oplus V$.
It combines partial spreads in the complement $V$ with colorings of powers of $q$-Johnson graphs.
This gives sufficient conditions for the existence of optimal Schubert subspace codes.
We also prove necessary conditions for constructions of this type, showing that independent sets and cliques in the graph naturally provide obstructions.

Our second construction is based on field reduction from extension fields.
Starting from suitable evasive or $h$-scattered $\Fq$-subspaces of $\F_{q^k}^r$, we obtain Schubert subspace codes whose parameters can be controlled explicitly.
In the scattered case, the size of the construction is computed exactly, and for $h=1$ we recover the construction of \cite{alfarano2024schubert}.

The paper is organized as follows.
Section~\ref{sec:prel} recalls the necessary background.
Section~\ref{sec:UV_construction} develops the direct-sum construction and studies both sufficient and necessary conditions for optimality.
Section~\ref{sec:scattered_generalized} presents the field reduction construction using evasive and scattered subspaces.
Section~\ref{sec:conclusions} draws concluding remarks and open problems.

\medskip

\section*{Acknowledgments}
This research has been partially supported by the Italian National Group for Algebraic and Geometric Structures and their Applications (GNSAGA - INdAM), by Università Italo Francese (UIF/UFI) via PHC Galileo 2026 - G26-260/54322VM and by the University of Naples Federico II, which provided funding through the ``Naples-Rennes Agreement on Scientific Cooperation''.
G. N. Alfarano is supported by the Agence Nationale de la Recherche through grant number ANR-24-CPJ1-0075-01.
A. Neri is supported by the INdAM - GNSAGA Project CUP E53C24001950001  ``Noncommutative polynomials in coding theory''.
B. Toesca is supported by the Swiss National Foundation through grant no. 212865.
The authors are thankful to Joachim Rosenthal and Vladislav Taranchuk for fruitful discussions.

\medskip

	\section{Preliminaries}\label{sec:prel}

   In this section, we introduce the necessary background for the rest of the paper. Throughout the paper, we will use the following notation.
    
\medskip

   \noindent{\textbf{Notation.}} Let $q$ be a prime power and $\F_q$ be the finite field of order $q$. For an integer $m\in\mathbb{N}$, we denote by $[m]:=\{1,\ldots,m\}$. Let $n,k,u,\ell,t$ be integers with $k,u\in[n]$, $0\leq\ell \leq \min\{k,u\}$ and $0\leq t\leq k-1$. We denote by $\Gr_q(k,n)$ the Grassmannian of $k$-dimensional subspaces of $\F_q^n$.

\medskip 
\subsection{Schubert subspace codes}

In \cite{alfarano2024schubert}, Schubert subspace codes have been considered as restrictions of subspace codes to Schubert varieties.

We first recall the definition of a Schubert variety in the Grassmannian $\Gr_q(k,n)$. Let $F=(V_1,V_2,\dots,V_n)$
be a \textbf{flag} of subspaces of $\Fq^n$, that is, a nested sequence
$$
\{0\}\subset V_1 \subset V_2 \subset \dots \subset V_n=\Fq^n
$$
such that $\dim(V_j)=j$ for every $j=1,2,\dots,n$. Let
$d=(d_1,\dots,d_k)$ be an ordered index set satisfying
$1\leq d_1<\dots<d_k\leq n$.
We say that a subspace $W\in\Gr_q(k,n)$ satisfies the \textbf{Schubert condition} $d$ with respect to the flag $F$ if
$\dim(W\cap V_{d_i})\geq i$ for every $i\in[k]$.

\begin{definition}
Let $F=(V_1,V_2,\dots,V_n)$ be a flag of subspaces of $\Fq^n$ and let $d=(d_1,\dots,d_k)$ be a Schubert condition. The \textbf{Schubert variety} associated with $F$ and $d$ is
$$
\Omega_d := \{W\in\Gr_q(k,n) \mid \dim(W\cap V_{d_i})\geq i, \; \forall i\in[k] \}.
$$
\end{definition}

We also recall that the Grassmannian $\Gr_q(k,n)$ is endowed with the \textbf{subspace distance}, defined for every $U,V\in\Gr_q(k,n)$ as
\begin{align*}
\dS(U,V)
&= \dim_{\Fq}(U+V)-\dim_{\Fq}(U\cap V)  \\
&= 2(k-\dim_{\Fq}(U\cap V)).
\end{align*}

A \textbf{constant-dimension subspace code} is a family of subspaces $\cS\subseteq \Gr_q(k,n)$. Its minimum subspace distance is defined as
$$
\dd_{\mathrm{S}}(\cS)
=
\min\{\dd_{\mathrm{S}}(U,V) \mid U,V\in \cS,\ U\ne V\}.
$$
A central question in the theory of constant-dimension subspace codes is the construction of codes in $\Gr_q(k,n)$ with the largest possible size for the given minimum subspace distance $d$. The largest size of such a code is denoted by $$A_q(n,k,d).$$ 
For the known values and bounds on $A_q(n,k,d)$, we refer the interested reader to \cite{heinlein2016tables} and references therein.

The largest possible minimum distance of a constant-dimension subspace code in $\Gr_q(k,n)$ is~$2k$.
This value is attained precisely when any two distinct codewords intersect trivially.
Such codes can exist when $n\geq 2k$ and geometrically define a \textbf{partial spread}.
If, moreover, for every vector $v\in\Fq^n$ there exists a codeword $U\in\cS$ containing it, the codewords form a \textbf{$k$-spread}. As shown in the seminal paper of Segre \cite{segre1964teoria}, a $k$-spread of $\bF_q^n$ exists if and only if $k$ divides $n$.
For this reason, these codes are called \textbf{spread codes} in the coding theory literature. 
Questions such as efficient decoding of this class of codes have been studied; see~\cite{go12,ma08p}.

In this work, we are interested in subspace codes of large size that live in a specific Schubert variety of $\Gr_q(k,n)$.
This point of view was introduced in \cite{alfarano2024schubert}, where \emph{Schubert subspace codes} were defined as constant-dimension subspace codes whose codewords are restricted to a fixed Schubert variety.

\begin{definition} \label{def:Schubert_code}
    Let $\Omega_d$ be a Schubert variety in the Grassmannian $\Gr_q(k,n)$. A family $\cS\subseteq\Omega_d$ equipped with the subspace distance is called a \textbf{Schubert subspace code} with respect to $\Omega_d$.
\end{definition}

We will focus on a special class of Schubert varieties. Let $U$ be a $u$-dimensional subspace of $\Fq^{n}$. Define
$$
\Omega_{U,\ell}:=\{W\in\Gr_q(k,n)\mid \dim(W\cap U)\geq \ell\}.
$$
For $\ell=1$, we simply write $\Omega_{U,1}=\Omega_U$.

The set $\Omega_{U,\ell}$ is a Schubert variety. Indeed, after fixing a flag
$$
\{0\}\subset V_1\subset \cdots \subset V_{n}=\Fq^{n}
$$
such that $V_u=U$, the condition $\dim(W\cap U)\geq \ell$ is one of the Schubert conditions defining a Schubert variety. In this case, $\Omega_{U,\ell}=\Omega_d$, where $d = (u -\ell+ 1,\ldots, u, n - k + \ell+1,\ldots, n)$.

\begin{definition} \label{def:lt-intersecting_set}
 We say that a Schubert subspace code $\mS\subseteq\Omega_{U,\ell}$ is an \textbf{$(\ell,t)$-intersecting set with respect to~$U$} if for every $S_i,S_j\in\mS$ such that $S_i\neq S_j$, it holds that
    $$
    \dim(S_i\cap S_j)\leq t.
    $$
\end{definition}

In words, we say that an $(\ell,t)$-intersecting set with respect to $U$ is a family of $k$-dimensional subspaces of $\Fq^{n}$ such that every element intersects $U$ in dimension at least $\ell$, while any two distinct elements intersect each other in dimension at most $t$.

Since the subspace distance on $\Gr_q(k,n)$ is given by $\dS(S_i,S_j)=2k-2\dim(S_i\cap S_j)$,
an $(\ell,t)$-intersecting set $\mS\subseteq\Omega_{U,\ell}$ is a Schubert subspace code of minimum distance at least $2k-2t$.
Conversely, a Schubert subspace code $\mS\subseteq\Omega_{U,\ell}$ with minimum distance at least $2k-2t$ is an $(\ell,t)$-intersecting set with respect to $U$.

\begin{remark}
    In \cite[Definition 3.2]{alfarano2024schubert}, a more restricted notion was considered. In particular, it was assumed that $k\mid n$ and that $u\leq \frac{n}{2}$. These assumptions were needed for the construction proposed there, which makes use of scattered subspaces. In this work, we aim to generalize the concept, so we separate the definition of an $(\ell,t)$-intersecting set from the extra assumptions that may be required for specific constructions.
\end{remark}

We denote by $m_q(n,k,u,\ell,t)$ the largest size of an $(\ell,t)$-intersecting set with respect to $U$. Notice that in \cite{alfarano2024schubert} the first two parameters inside $m_q$ were $k$ and $r=\frac{n}{k}$, but here we use this more general notation since we no longer require $n$ to be a multiple of $k$.
Our goal is to estimate the quantity $m_q(n,k,u,\ell,t)$ and to construct codes that achieve such maximum cardinality.

For the case $t\leq\ell-1$, we give the following upper bound on $m_q(n,k,u,\ell,t)$, which generalizes the one obtained in \cite[Proposition 4.4]{alfarano2024schubert} for the case $t=\ell-1$.
\begin{proposition}\label{prop:upper_bound}
 If $t\le \ell-1$, then
 $$m_q(n,k,u,\ell,t)\le A_q(u,\ell,2(\ell-t)).$$
 In particular, if $t=\ell-1$, then
 $$m_q(n,k,u, \ell,\ell-1)\le \qbin{u}{\ell}{q}.$$
\end{proposition}

\begin{proof}
 Let us consider an $(\ell,t)$ -intersecting set $\mS\subseteq \Omega_{U,\ell}$. Define the map 
 $$ \begin{array}{rcc}\Phi:\mS & \longrightarrow & \bigcup\limits_{i=\ell}^k\Gr_q(i,u)\\
 S & \longmapsto & S\cap U. \end{array}$$
 Since $t\leq\ell-1$, we show that the map $\Phi$ is injective.
 Clearly, since $\cS$ is an $(\ell,t)$-intersecting set, for every $S\in\mS$ we have
 $$\dim(\Phi(S))\ge \ell.$$
 Moreover, $\Phi(\mS)$ is a set of subspaces such that, for every two distinct $S_1,S_2\in \mS$, \begin{align*}
\dim\big(\Phi(S_1)\cap\Phi(S_2)\big)&=\dim(S_1\cap S_2\cap U)\le \dim(S_1\cap S_2)\le t<\ell.
 \end{align*}
 Hence, $\Phi(S_1)\neq\Phi(S_2)$ for every $S_1\neq S_2$ and $\Phi$ is injective.
 
At this point, for every $S\in \mS$ choose a subspace $T_S\subseteq \Phi(S)$ such that $\dim(T_S)=\ell$. By the intersection property, we have $\dim(T_{S_1}\cap T_{S_2})\le t$ for every $S_1,S_2 \in \mS$ with $S_1\neq S_2$, and therefore the set
$$\{T_S\,:\,S\in \mS\}$$
is a constant-dimension subspace code in $\Gr_q(\ell,u)$ of size $|\mS|$ and minimum distance at least $2(\ell-t)$. Thus, we conclude that $|\mS|\le A_q(u,\ell,2(\ell-t))$.

The second part of the statement follows from the fact that, when $t=\ell-1$, then $2(\ell-t)=2$ and $$A_q(u,\ell,2)=\qbin{u}{\ell}{q}.$$

\end{proof}

\begin{remark}
\cref{prop:upper_bound} also clarifies the relation between the problem considered here
and standard block constructions of constant-dimension codes, such as coset
constructions; see e.g. \cite{heinlein2017coset}. 
Coset constructions are general methods for producing
large constant-dimension codes in the full Grassmannian $\Gr_q(k,n)$. They are typically based on a decomposition of the ambient space into coordinate blocks,
together with rank-metric ingredients, and their aim is to obtain good lower
bounds on the unrestricted quantity $A_q(n,k,d)$. In particular, they are designed to control the dimension of the intersection of two codewords in $\F_q^n$.

The situation considered in this paper is different. Here one fixes a subspace
$U\subseteq \F_q^n$ and requires all codewords to lie in the Schubert variety
\[
    \Omega_{U,\ell}
    =
    \{W\in \Gr_q(k,n): \dim(W\cap U)\ge \ell\}.
\]
Thus, the intersections with the fixed subspace $U$ impose an additional
constraint. \cref{prop:upper_bound} shows that, when $t\leq\ell-1$, these intersections
already yield the upper bound
\[
    m_q(n,k,u,\ell,t)
    \le
    A_q(u,\ell,2(\ell-t)).
\]
Consequently, any construction attaining this bound must have the property that
the $\ell$-dimensional subspaces extracted from the intersections with $U$ form
an optimal constant-dimension subspace code in $\Gr_q(\ell,u)$.

This requirement is not part of the usual coset constructions. Such
constructions may have a block form which, after a suitable choice of coordinates,
forces the codewords to meet a fixed
coordinate subspace in a prescribed dimension. In this sense, some of the
coset constructions may be viewed as Schubert subspace codes.
However, membership in a Schubert variety is not enough to attain the bound of \cref{prop:upper_bound}.
Equality requires that the map
\[
    W \longmapsto T_W\subseteq W\cap U,
    \qquad \dim(T_W)=\ell,
\]
realizes an optimal code of size $A_q(u,\ell,2(\ell-t))$ inside $U$. Standard
coset constructions are not designed to impose this optimal
intersection pattern with a fixed subspace, and indeed, to our knowledge, they do not meet the bound of \cref{prop:upper_bound}.
\end{remark}

\begin{remark}
    If $\ell=1$, $t=0$ and $u\leq\frac{n}{2}$, it was shown in \cite{alfarano2024schubert} that, whenever $k|n$ and there exists a $u$-dimensional scattered $\Fq$-subspace of $\F_{q^k}^{\frac{n}{k}}$, there is a construction that achieves the bound of \cref{prop:upper_bound}. In the case $\ell=k$ and $t=k-1$, one can also trivially achieve the bound by selecting all $k$-dimensional subspaces contained in $U$.
\end{remark}

Our goal is to find more constructions of $(\ell,t)$-intersecting sets whose size achieves the upper bound of \cref{prop:upper_bound} for different values of $\ell$ and $t$, with $t\leq \ell-1$. In this case, following the proof of the upper bound, one can hope to construct optimal Schubert subspace codes by assigning to each $\ell$-dimensional subspace of $U$ a suitable $k$-dimensional codeword in~$\Omega_{U,\ell}$.

Thus, throughout the paper, the terminology of $(\ell,t)$-intersecting sets is used as a convenient way to describe Schubert subspace codes contained in the Schubert variety $\Omega_{U,\ell}$. More precisely, every family considered in the sequel is a constant-dimension code in $\Gr_q(k,n)$ whose codewords satisfy the Schubert condition $\dim(W\cap U)\ge \ell$, and whose minimum subspace distance is controlled by the parameter $t$. Therefore, constructing large $(\ell,t)$-intersecting sets is exactly the problem of constructing large Schubert subspace codes in $\Omega_{U,\ell}$ with minimum distance at least $2k-2t$. In particular, when the upper bound of \cref{prop:upper_bound} is attained, the corresponding family is an optimal Schubert subspace code with respect to this Schubert variety and distance constraint.

\subsection{\emph{q}-Johnson graphs and their powers}\label{sec:q-Johnson}

In this subsection, we introduce the $q$-Johnson graph and some results about its chromatic number.
We first recall the definition of the classical \textit{Johnson graph}.

\begin{definition}
The \textbf{Johnson graph} $J(n,k)$ is defined as follows. The vertices of the graph are the subsets of cardinality $k$ of a set with $n$ elements. There is an edge between two vertices if the corresponding subsets intersect in a subset of cardinality $k-1$.
\end{definition}

The $q$-analogue for subspaces is known as the \textit{$q$-Johnson graph} or \textit{Grassmann graph}.

\begin{definition}
    Given a finite field of order $q$, the \textbf{$q$-Johnson graph} $J_q(n,k)$ is defined as follows. The vertices of the graph are the $k$-dimensional subspaces of an $n$-dimensional vector space over $\Fq$. There is an edge between two vertices if the corresponding subspaces intersect in a subspace of dimension $k-1$.
\end{definition}

For our construction, it is useful to know the \textbf{chromatic number} of this graph, i.e., the
minimum number of colors needed to color the vertices of the graph in such a way that two
adjacent vertices are always of different colors. We denote the chromatic number of a graph $G$
by $\chi(G)$.

We recall the following classical bound from \cite{brooks1941colouring} on the chromatic number
of a graph.

\begin{theorem}[Brooks' theorem]\label{thm:brooks}
Given a connected graph $G$ with maximum degree $\Delta$, its chromatic number is bounded by
$$
\chi(G)\leq \Delta+1.
$$
Moreover, equality holds if and only if $G$ is a complete graph or an odd cycle.
\end{theorem}

\begin{corollary}\label{cor:X(q-Johnson)_from_Brooks}
If $2\leq k\leq n-2$, then the chromatic number of the $q$-Johnson graph $J_q(n,k)$ satisfies
$$
\chi(J_q(n,k)) \leq q\cdot\qbin{k}{1}{q}\cdot\qbin{n-k}{1}{q}.
$$
\end{corollary}

\begin{proof}
Note that the $q$-Johnson graph is regular of degree
$$
\Delta=q\cdot\qbin{k}{1}{q}\cdot\qbin{n-k}{1}{q}.
$$
Indeed, for each vertex, there are
$$
\qbin{k}{k-1}{q}=\qbin{k}{1}{q}
$$
ways to choose one of its $(k-1)$-dimensional subspaces, and for each of them there are
$$
\left(\qbin{n-(k-1)}{1}{q}-1\right)
=
q\cdot\qbin{n-k}{1}{q}
$$
ways to complete it to a $k$-dimensional space different from the original one. Since
$2\leq k\leq n-2$, the graph $J_q(n,k)$ is not complete. Moreover, it is not an odd cycle.
Hence, by \cref{thm:brooks}, its chromatic number is bounded by its maximum degree~$\Delta$.
\end{proof}

A stronger bound was recently proved in \cite{dhaeseleer2025chromatic}. 

\begin{theorem}[\cite{dhaeseleer2025chromatic}]\label{thm:bound_JozefienVlad1}
The chromatic number of the $q$-Johnson graph $J_q(n,k)$ is bounded by
$$\chi(J_q(n,k))
\leq
\qbin{n}{1}{q}.
$$
\end{theorem}

We will also need chromatic bounds for powers of the $q$-Johnson graph. Recall that, for a graph $G$ and an integer $s\geq 1$, the \textbf{$s$-th power $G^s$} is the graph with the same vertex set as $G$, where two distinct vertices are adjacent if their geodesic distance in $G$ is at most $s$.
In the case of the $q$-Johnson graph, two vertices of $J_q(n,k)^s$ are adjacent if and only if
the corresponding $k$-dimensional subspaces intersect in dimension at least $k-s$. Indeed, the geodesic
distance between two $k$-dimensional subspaces $A$ and $B$ in the $q$-Johnson graph is
$$
k-\dim(A\cap B)=\frac{1}{2}\dS(A,B).
$$
Thus, increasing the power $s$ makes the adjacency relation stronger. In particular, once $s$ is at least
the diameter of $J_q(n,k)$, the graph $J_q(n,k)^s$ is complete. Since the diameter is
$\min\{k,n-k\}$, this happens for $s\geq \min\{k,n-k\}$.

For the chromatic number of the $s$-th power of the $q$-Johnson graph, the degree is given by
$$
\Delta(J_q(n,k)^s)
=
\sum_{i=1}^{\min\{s,k,n-k\}} q^{i^2}\qbin{k}{k-i}{q}\qbin{n-k}{i}{q}.
$$
Therefore, \cref{thm:brooks} gives
$$
\chi(J_q(n,k)^s)
\leq
\Delta(J_q(n,k)^s)+1.
$$
If $s<\min\{k,n-k\}$, then  $J_q(n,k)^s$ is neither complete nor an odd cycle, hence the stronger bound
$$
\chi(J_q(n,k)^s)
\leq
\Delta(J_q(n,k)^s)
$$
holds.

In \cite{dhaeseleer2026chromatic}, the authors study the chromatic number of $q$-Johnson
graphs and their powers using maximum rank distance (MRD) codes. We will use the following upper bound.

\begin{theorem}[{\cite[Theorem 1.3]{dhaeseleer2026chromatic}}]\label{thm:bound_JozefienVlad2}
Let $1\leq s< \min\{k,n-k\}$. The chromatic number of the $s$-th power of the $q$-Johnson
graph $J_q(n,k)$ is bounded by
$$\qbin{\max\{k,n-k\}+s}{s}{q}\le 
\chi(J_q(n,k)^s)
\leq
(1+o(1))n^s
\qbin{\max\{k,n-k\}+s}{s}{q},
$$
where $o(1)\to 0$ as $n\to \infty$ and  does not depend on $q$.
\end{theorem}

In particular, in \cite{dhaeseleer2026chromatic} they derive that for fixed $n,k,s$ and growing
$q$, the chromatic number has asymptotic order
\begin{equation}\label{eq:asymptotic_chromatic}
\chi(J_q(n,k)^s)=\Theta\left(q^{s\max\{n-k,k\}}\right).
\end{equation}

For $s=1$, we will use the sharper bound of \cref{thm:bound_JozefienVlad1}. For larger values
of $s$, the bound of \cref{thm:bound_JozefienVlad2} will be useful for the construction in the case $t\leq\ell-1$. In that application, the graph will be $J_q(u,\ell)^{k-t-1}$, so the graph is complete when $k-t-1\geq \min\{\ell,u-\ell\}$ and should be treated separately.

\section{A direct-sum construction}\label{sec:UV_construction}
In this section, we focus on the general case $t\le \ell-1$, and we assume that
$1\leq \ell\leq k-1$. Recall that, under these assumptions, the upper bound of
\cref{prop:upper_bound} holds, and we have
$$
m_q(n,k,u,\ell,t)\leq A_q(u,\ell,2(\ell-t)).
$$

Let $U$ be a $u$-dimensional subspace of $\Fq^n$, and let $V$ be a complement of $U$ in
$\Fq^n$, so that $U\cap V=\{0\}$ and $U\oplus V=\Fq^n$. We also assume that
$k-\ell\leq n-u$, so that $V$ contains subspaces of dimension $k-\ell$.

We will construct $(\ell,t)$-intersecting sets $\mS$ in the Schubert variety
$\Omega_{U,\ell}$ satisfying the stronger condition $\dim(S\cap U)=\ell$ for every
$S\in\mS$. Whenever such a construction has cardinality $A_q(u,\ell,2(\ell-t))$, it is
\emph{optimal} also in the full Schubert variety $\Omega_{U,\ell}$, by \cref{prop:upper_bound}. We will start with an optimal constant-dimension subspace code 
$$\cA\subseteq \{A\subseteq U : \dim(A)=\ell\}=\Gr_q(\ell,u)$$ 
of minimum subspace distance $2(\ell-t)$ and size $|\cA|=A_q(u,\ell,2(\ell-t))$.

Note that, when $t=\ell-1$, we have that $\cA=\Gr_q(\ell,u)$, that is, the constant-dimension subspace code $\cA$ is the full Grassmannian of $\ell$-dimensional subspaces of $U$.
Throughout the paper, we will first state our results in the general case $t\leq\ell-1$, and then analyze the case $t=\ell-1$ in which we can give more explicit bounds.

We define
\begin{align*}
\cB&=\{B\subseteq V : \dim(B)=k-\ell\}=\Gr_q(k-\ell,n-u).
\end{align*}
For every $A\in\cA$ and $B\in\cB$, we have $\dim(A+B)=k$, since $U$ and $V$ intersect
trivially.

The direct-sum constructions considered in this section are obtained from maps
$$
\varphi:\cA\longrightarrow \cB.
$$
To such a map we associate the family
$$
\mS_\varphi:=\{A\oplus \varphi(A): A\in\cA\}.
$$
Every element of $\mS_\varphi$ has dimension $k$ and intersects $U$ exactly in dimension
$\ell$. This means that $\cS_\varphi$ is a subspace code of constant dimension $k$ and, in particular, a Schubert subspace code contained in the Schubert variety $\Omega_{U,\ell}$.
Moreover, it has cardinality
$$
|\mS_\varphi|=|\cA|=A_q(u,\ell,2(\ell-t)).
$$
Hence, by \cref{prop:upper_bound}, if $\mS_\varphi$ is an $(\ell,t)$-intersecting set,
then it is optimal.
This section is therefore devoted to finding necessary and sufficient conditions for $\cS_\varphi$ to be an $(\ell,t)$-intersecting set.

The following lemma is the main observation behind the construction.

\begin{lemma}\label{lem:dim(S1capS2)}
For $i=1,2$, let $A_i\in\cA$, $B_i\in\cB$, and $S_i=A_i\oplus B_i$. Then
$$
\dim(S_1\cap S_2)=\dim(A_1\cap A_2)+\dim(B_1\cap B_2).
$$
\end{lemma}

\begin{proof}
Since $A_1+A_2\subseteq  U$ and $B_1+B_2\subseteq  V$, we have
$$
(A_1+A_2)\cap(B_1+B_2)=\{0\}.
$$
Hence,
$$
\dim(A_1+A_2+B_1+B_2)
=
\dim(A_1+A_2)+\dim(B_1+B_2).
$$
Moreover, $S_1+S_2=A_1+A_2+B_1+B_2$. Therefore,
\begin{align*}
\dim(S_1\cap S_2)
&=\dim(S_1)+\dim(S_2)-\dim(S_1+S_2)\\
&=2k-\dim(A_1+A_2+B_1+B_2)\\
&=2k-\dim(A_1+A_2)-\dim(B_1+B_2)\\
&=2k-\left(2\ell-\dim(A_1\cap A_2)\right)
-\left(2(k-\ell)-\dim(B_1\cap B_2)\right)\\
&=\dim(A_1\cap A_2)+\dim(B_1\cap B_2).
\end{align*}
\end{proof}

Thus, finding an optimal direct-sum construction amounts to finding a map
$\varphi:\cA\to\cB$ such that, for all distinct $A_1,A_2\in\cA$, we have
$$
\dim(A_1\cap A_2)+
\dim(\varphi(A_1)\cap\varphi(A_2))
\leq t.
$$

In the following subsections, we study sufficient and necessary conditions for the existence of $\varphi:\cA\to\cB$ such that $\mS_\varphi$ is 
an $(\ell,t)$-intersecting set with cardinality
$$
|\mS|=A_q(u,\ell,2(\ell-t)).
$$
Recall that, by \cref{prop:upper_bound}, this is the maximum possible cardinality for an
intersecting set with such parameters.

\subsection{Sufficient conditions}\label{sec:sufficient_conditions}

We start with a simple construction. Recall that a \textbf{partial $m$-spread} in $\F_q^n$ is a family
of $m$-dimensional subspaces of $\F_q^n$ which are pairwise disjoint. Finding large partial
spreads in $\F_q^n$ has proven to be a very difficult mathematical problem. In the 1970s,
Beutelspacher discovered a construction of large partial spreads which leads to the following
theorem; see \cite{beutelspacher1975partial}.

\begin{theorem}[{\cite[Theorem 4.2]{beutelspacher1975partial}}]\label{thm:beutelspacher_partial_spread}
Let $m$ and $n$ be positive integers with $m\leq n$, and let $r$ be the remainder of $n$ modulo $m$.
Then $\Fq^n$ contains a partial $m$-spread of size
$$
\frac{q^n-q^{m+r}}{q^m-1}+1.
$$
\end{theorem}

\begin{remark}
If $m$ divides $n$, then the remainder $r$ in
\cref{thm:beutelspacher_partial_spread} is zero, and the bound becomes
$$
\frac{q^{n}-q^{m}}{q^{m}-1}+1
=
\frac{q^{n}-1}{q^{m}-1}.
$$
In this case, $\bF_q^n$ admits a complete $m$-spread, retrieving the existence result of Segre
\cite{segre1964teoria}.
\end{remark}

For our construction, we will need partial spreads in the space $V$, which by definition is an $(n-u)$-space such that $U\oplus V=\bF_q^n$.

\begin{proposition}\label{prop:partial_spread_simple}
Suppose that $V$ contains a partial $(k-\ell)$-spread $\mP$ such that
$|\mP|\geq A_q(u,\ell,2(\ell-t))$. Then
$$
m_q(n,k,u,\ell,t)=A_q(u,\ell,2(\ell-t)).
$$
\end{proposition}

\begin{proof}
Let $M:=A_q(u,\ell,2(\ell-t))$, and let $\mathcal A$ be a constant-dimension subspace code of minimum subspace distance $2(\ell-t)$ and size $A_q(u,\ell,2(\ell-t))$. We can order the spaces in $\cA$ as $A_1,\ldots,A_M$.
By assumption, there exist pairwise disjoint subspaces $B_1,\ldots,B_M\in\mP$.
Define the function $\varphi:\cA\to\cB$
$$
\varphi(A_i):=B_i
$$
and the constant-dimension subspace code
$$
\mS_\varphi=\{A_i\oplus B_i : i=1,\dots,M\}.
$$
For every element $S_i=A_i\oplus B_i$, we have $\dim(S_i\cap U)=\dim(A_i)=\ell$. Moreover,
by \cref{lem:dim(S1capS2)}, for every $i\neq j$ we have
$$
\dim(S_i\cap S_j)
=
\dim(A_i\cap A_j)+\dim(B_i\cap B_j)
=
\dim(A_i\cap A_j)
\leq t.
$$
Hence, $\mS_\varphi$ is an $(\ell,t)$-intersecting set with respect to $U$ of cardinality
$A_q(u,\ell,2(\ell-t))$. By \cref{prop:upper_bound}, this cardinality is maximum.
\end{proof}

The construction in \cref{prop:partial_spread_simple} assigns a different element of the
partial spread to each element of $\cA$. This is enough to make all the contributions coming
from the second component disjoint, but it is often too restrictive. We can do better by
allowing different elements of $\cA$ to be completed with the same subspace of $V$.

By \cref{lem:dim(S1capS2)}, this is possible only when the corresponding subspaces of $U$
intersect in a sufficiently small dimension. More precisely, if two distinct elements
$A_1,A_2\in\cA$ are completed with the same $B\in\cB$, then we need $\dim(A_1\cap A_2)+(k-\ell)\leq t$,
or equivalently
$$
\dim(A_1\cap A_2)<\ell+t+1-k.
$$
This is precisely the condition that $A_1$ and $A_2$ are not adjacent in the power $J_q(u,\ell)^{k-t-1}$.

The idea is to color the vertices of $J_q(u,\ell)^{k-t-1}$ and assign the same element of a partial $(k-\ell)$-spread in $V$ to all vertices of the same color. Vertices with different
colors are assigned disjoint subspaces of~$V$. This is formally explained in \cref{prop:chi_small_is_sufficient}.

Notice that, when $\ell+t+1-k\leq 0$, the graph $J_q(u,\ell)^{k-t-1}$ is complete, so every vertex belongs to a different color class.
Thus, in this range, the coloring construction below reduces to the construction from \cref{prop:partial_spread_simple}, since two different elements of $\cA$ are always mapped to disjoint subspaces of $V$.

{\begin{theorem}\label{prop:chi_small_is_sufficient}
Suppose that $V$ contains a partial $(k-\ell)$-spread $\mP$ such that
$|\mP|\geq \chi(J_q(u,\ell)^{k-t-1})$. Then
$$
m_q(n,k,u,\ell,t)=A_q(u,\ell,2(\ell-t)).
$$
\end{theorem}

\begin{proof} Fix a constant-dimension subspace code $\cA$ in $\Gr_q(\ell,u)$ of minimum subspace distance $2(\ell-t)$ and cardinality $|\cA|=A_q(u,\ell,2(\ell-t))$.
Consider the subgraph $\Gamma:=J_q(u,\ell)^{k-t-1}[\cA]$ induced from $J_q(u,\ell)^{k-t-1}$ by the vertices in $\cA$. Let the chromatic number be $\chi:=\chi(\Gamma)$, and fix a coloring of $\Gamma$ with $\chi$ colors.
Partition $\cA=\cA_1\cup\cdots\cup\cA_\chi$, where $\cA_j$ is the set of vertices of color $j$.
Since $|\mP|\geq \chi(J_q(u,\ell)^{k-t-1})\geq \chi$, we can choose pairwise disjoint subspaces $B_1,\ldots,B_\chi\in\mP$.
Define $\varphi:\cA\to\cB$ by
$$
\varphi(A):=B_j \qquad \text{if } A\in\cA_j.
$$
Then
$$
\mS_\varphi
=
\bigcup_{j=1}^{\chi}\{A\oplus B_j : A\in\cA_j\}.
$$
Clearly $|\mS_\varphi|=|\cA|=A_q(u,\ell,2(\ell-t))$, and every element
$S\in\mS_\varphi$ satisfies $\dim(S\cap U)=\ell$ so $\cS_\varphi$ is a Schubert subspace code.

We prove that $\mS_\varphi$ is an $(\ell,t)$-intersecting set. Let $S_1=A_1\oplus \varphi(A_1)$ and $S_2=A_2\oplus \varphi(A_2)$ be two distinct elements of $\mS_\varphi$ and we study the dimension of their intersection.

Suppose first that $A_1$ and $A_2$ lie in different color classes. Then
$\varphi(A_1)$ and $\varphi(A_2)$ are distinct elements of the partial spread $\mP$, and hence
$\varphi(A_1)\cap \varphi(A_2)=\{0\}$. By \cref{lem:dim(S1capS2)} we get
$$
\dim(S_1\cap S_2)=\dim(A_1\cap A_2)\leq t.
$$

Suppose now that $A_1$ and $A_2$ lie in the same color class. Then they are not adjacent in
$\Gamma$. Hence,
$$
\dim(A_1\cap A_2)<\ell-(k-t-1)=\ell+t+1-k.
$$
Moreover, $\varphi(A_1)=\varphi(A_2)$ has dimension $k-\ell$. By \cref{lem:dim(S1capS2)},
$$
\dim(S_1\cap S_2)
=
\dim(A_1\cap A_2)+\dim(\varphi(A_1))
<
(\ell+t+1-k)+(k-\ell)
=
t+1.
$$
Therefore, $\dim(S_1\cap S_2)\leq t$. Thus, $\mS_\varphi$ is an
$(\ell,t)$-intersecting set with respect to $U$ of cardinality
$A_q(u,\ell,2(\ell-t))$. By \cref{prop:upper_bound}, this cardinality is maximum.
\end{proof}

\begin{remark}\label{rem:strengthening_chi_subgraph}
    Note that \cref{prop:chi_small_is_sufficient} can be strengthened for $t<\ell-1$ by requiring that there exists a partial $(k-\ell)$-spread $\cP$ in $V$ such that $\chi(\Gamma)\le |\cP|$, $\Gamma=J_q(u,\ell)^{k-t-1}[\cA]$ and $\cA$ is a constant-dimension code in $\Gr_q(\ell,u)$  of minimum subspace distance $2(\ell-t)$ and cardinality $A_q(u,\ell,2(\ell-t))$. However, since the chromatic number of the induced subgraph depends on the construction of $\cA$, we stated the result of \cref{prop:chi_small_is_sufficient} more generally by just considering the chromatic number of the entire graph $J_q(u,\ell)^{k-t-1}$.
    Due to known estimates on the chromatic number of $J_q(u,\ell)^{k-t-1}$, this result can be easier to apply, as we will see in \cref{cor:ell=k-1_} and \cref{cor:asymptotic_general_ell}.
\end{remark}

\begin{corollary}\label{cor:partial_spread_condition}
Let $r$ be the remainder of $n-u$ modulo $k-\ell$. If
$$
\chi(J_q(u,\ell)^{k-t-1})
\leq
\frac{q^{n-u}-q^{k-\ell+r}}{q^{k-\ell}-1}+1,
$$
then
$$
m_q(n,k,u,\ell,t)=A_q(u,\ell,2(\ell-t)).
$$
\end{corollary}

\begin{proof}
By \cref{thm:beutelspacher_partial_spread}, the vector space $V$, which has dimension $n-u$,
contains a partial $(k-\ell)$-spread of size
$$
\frac{q^{n-u}-q^{k-\ell+r}}{q^{k-\ell}-1}+1.
$$
Hence, the result follows from \cref{prop:chi_small_is_sufficient}.
\end{proof}

We now use the results from \cref{sec:q-Johnson} to estimate the parameters for which the condition on the chromatic number holds.

Note that in the very special case in which  $\ell=k-1$ and $t=k-2$, we have $k-t-1=1$ and hence the graph we consider in \cref{prop:chi_small_is_sufficient} is simply $J_q(u,\ell)$, whose chromatic number can be better estimated via \cref{thm:bound_JozefienVlad1}. Moreover, in this case, the elements of $\cB$ are the
$1$-dimensional subspaces of $V$, and any two distinct elements of $\cB$ intersect trivially. Hence, we obtain the following result.

\begin{corollary}\label{cor:ell=k-1_}
Suppose $\ell=k-1$. If $u\leq \frac{n}{2}$, then
$$
m_q(n,k,u,k-1,k-2)=\qbin{u}{k-1}{q}.
$$
\end{corollary}

\begin{proof}
If $\ell=k-1$, then $k-\ell=1$. By \cref{thm:bound_JozefienVlad1}, we have
$\chi(J_q(u,k-1))\leq \qbin{u}{1}{q}$. Moreover, since $u\leq \frac n2$, we have
$\qbin{u}{1}{q}\leq \qbin{n-u}{1}{q}$. The set of all $1$-dimensional subspaces of $V$ is a
(partial) $1$-spread of size $\qbin{n-u}{1}{q}$. Hence, the claim follows from
\cref{prop:chi_small_is_sufficient}.
\end{proof}

We can also obtain asymptotic sufficient conditions for general $\ell$ by using the bound on the
chromatic number of powers of the $q$-Johnson graph. We first consider the non-complete case,
where the coloring argument can improve on the construction in
\cref{prop:partial_spread_simple}.

\begin{theorem}\label{cor:asymptotic_general_ell}
Assume that $k-t-1<\min\{\ell,u-\ell\}$.
For $q$ large enough, if
$$
(k-t-1)\max\{u-\ell,\ell\}< n-u-(k-\ell),
$$
then
$$
m_q(n,k,u,\ell,t)=A_q(u,\ell,2(\ell-t)).
$$
\end{theorem}

\begin{proof}
We apply \cref{thm:bound_JozefienVlad2} to the $(k-t-1)$-th power of the $q$-Johnson graph
$J_q(u,\ell)$. In particular, \eqref{eq:asymptotic_chromatic} gives
$$
\chi(J_q(u,\ell)^{k-t-1})=\Theta(q^{(k-t-1)\max\{u-\ell,\ell\}})
$$
for fixed $u,k,\ell$ and $q$ large.

On the other hand, by  \cref{thm:beutelspacher_partial_spread} applied to $V$ and to $m=k-\ell$, we obtain a partial $(k-\ell)$-spread in $V$ of size
$$
\frac{q^{n-u}-q^{k-\ell+r}}{q^{k-\ell}-1}+1,
$$
where $r$ is the remainder of $n-u$ modulo $k-\ell$. This quantity has order
$q^{n-u-(k-\ell)}$.

Therefore, if
$$
(k-t-1)\max\{u-\ell,\ell\}< n-u-(k-\ell),
$$
then the size of the available partial spread is asymptotically larger than the chromatic upper
bound. Hence, for $q$ large enough, there exists a partial $(k-\ell)$-spread $\mP$ in $V$ such
that
$$
|\mP|\geq \chi(J_q(u,\ell)^{k-t-1}).
$$
The result follows from  \cref{cor:partial_spread_condition}.
\end{proof}

\begin{remark}
The assumption $k-t-1<\min\{\ell,u-\ell\}$ in \cref{cor:asymptotic_general_ell} ensures that the graph
$J_q(u,\ell)^{k-t-1}$ is not complete. If $k-t-1\geq \min\{\ell,u-\ell\}$, then
$J_q(u,\ell)^{k-t-1}$ is the complete graph, and therefore
$$
\chi(J_q(u,\ell)^{k-t-1})=\qbin{u}{\ell}{q}.
$$
If this happens, we distinguish two cases:
\begin{enumerate}
    \item If $t=\ell-1$, the coloring construction of \cref{prop:chi_small_is_sufficient} does not
 improve on the construction of \cref{prop:partial_spread_simple}: one needs a
different element of the partial $(k-\ell)$-spread in $V$ for each element of $\cA$.
Consequently, in the complete case, \cref{prop:chi_small_is_sufficient} reduces to the condition
that $V$ contains a partial $(k-\ell)$-spread of size at least $\qbin{u}{\ell}{q}$. 
\item If $t<\ell-1$, \cref{prop:partial_spread_simple} gives a tighter result than \cref{prop:chi_small_is_sufficient}: however, using the stronger considerations of \cref{rem:strengthening_chi_subgraph} about the subgraph $\Gamma=J_q(u,\ell)^{k-t-1}[\cA]$, we have that $\Gamma$ is the complete graph on $A_q(u,\ell,2(\ell-t))$ vertices, and hence  we retrieve again \cref{prop:partial_spread_simple}.
\end{enumerate}

\end{remark}

\begin{remark}
The condition in \cref{cor:asymptotic_general_ell} can be rewritten more explicitly.

Assume first that $J_q(u,\ell)^{k-t-1}$ is not complete.
If $u\geq 2\ell$, then $\max\{u-\ell,\ell\}=u-\ell$. The condition becomes
$(k-t-1)(u-\ell)< n-u-(k-\ell).$
Equivalently,
$(k-t)u<n+\ell(k-t)-k,$
and hence
$$
u<\ell+\frac{n-k}{k-t}.
$$
If $u\leq 2\ell$, then $\max\{u-\ell,\ell\}=\ell$. The condition becomes
$(k-t-1)\ell<n-u-(k-\ell)$,
which is equivalent to
$$
u<n-(k-\ell)-(k-t-1)\ell.
$$

If $J_q(u,\ell)^{k-t-1}$ is complete, then $k-t-1\ge \min\{u-\ell,\ell\}$. In this case, 
$$A_q(u,\ell,2(\ell-t))=\Theta(q^{\max\{u-\ell,\ell\}(\min\{u-\ell,\ell\}-(\ell-t)+1)});$$
see e.g. \cite{heinlein2016tables}. 
Hence, for $q$ large, the sufficient condition becomes
$$
\max\{u-\ell,\ell\}(\min\{u-\ell,\ell\}-(\ell-t)+1)
<
n-u-k+\ell.$$
Equivalently, if $u\le 2\ell$, this condition is
$$
(\ell+1)u < n-k+2\ell^2-\ell t,$$
whereas if $u\ge 2\ell$, it is
$$
u < \ell+\frac{n-k}{t+2}.$$
\end{remark}

The results above give sufficient conditions for the direct-sum construction to attain the upper bound of \cref{prop:upper_bound}.
We now turn to study necessary conditions for constructions of this form.

\subsection{Necessary conditions}\label{sec:necessary_conditions}

As in the previous subsection, recall that we define $\cA\subseteq \Gr_q(\ell,u)$ to be a constant-dimension subspace code of minimum subspace distance $2(\ell-t)$ and size $A_q(u,\ell,2(\ell-t))$, and $\cB=\Gr_q(k-\ell,n-u)$.
Recall also that, given a map $\varphi:\cA\to\cB$, we can associate to it the family $\cS_\varphi=\{A\oplus\varphi(A):A\in\cA\}$.
$\cS_\varphi$ is always a Schubert subspace code of constant dimension $k$ contained in the Schubert variety $\Omega_{U,\ell}$.

The results in \cref{sec:sufficient_conditions} give sufficient conditions for $\cS_\varphi$ to not only be a Schubert subspace code, but also an $(\ell,t)$-intersecting set with optimal cardinality.
More precisely, they show that the upper bound of \cref{prop:upper_bound} is attained whenever the complement $V$ contains a sufficiently large partial $(k-\ell)$-spread, possibly after grouping the elements of $\cA$ according to a coloring of the graph $J_q(u,\ell)^{k-t-1}$.

We now look at the converse direction. Namely, we study necessary conditions for the existence of a map
$\varphi:\cA\to\cB$ giving an optimal $(\ell,t)$-intersecting set of the form
$$
\mS_\varphi=\{A\oplus\varphi(A):A\in\cA\}.
$$
This allows us to understand which constraints must be satisfied when assigning to each
$\ell$-dimensional subspace of $U$ a $(k-\ell)$-dimensional subspace of $V$.

Throughout this subsection, we assume that $\varphi:\cA\to\cB$ is such that $\mS_\varphi$
is an $(\ell,t)$-intersecting set, with $t\leq \ell-1$, and look for necessary conditions between the parameters.

We first study what happens when two spaces in $\cA$ have the same image under $\varphi$.

\begin{lemma}\label{prop:phi_injective}
    If, for some $B\in\cB$, there exist two distinct $A_1,A_2\in\varphi^{-1}(B)$, then
    $$
    \dim(A_1\cap A_2)\leq {\ell+t-k}.
    $$
\end{lemma}
\begin{proof}
    Let $A_1,A_2\in \varphi^{-1}(B)$, with $A_1\neq A_2$. Since
    $\varphi(A_1)=\varphi(A_2)=B$, by \cref{lem:dim(S1capS2)} we have
    $$
    \dim\big((A_1\oplus B)\cap(A_2\oplus B)\big)
    =
    \dim(A_1\cap A_2)+\dim(B).
    $$
    Since $\dim(B)=k-\ell$ and $\mS_\varphi$ is an $(\ell,{t})$-intersecting set, we get
    $$
    \dim(A_1\cap A_2)+k-\ell \leq {t}.
    $$
    Hence, $\dim(A_1\cap A_2)\leq {\ell+t-k}$.
\end{proof}

As an immediate consequence, in the range $k{>\ell+t}$ we get a simple necessary condition.

\begin{corollary}\label{cor:injective_necessary_condition}
Assume that $k {>\ell+t}$.
Then $\varphi$ needs to be injective and
\begin{equation}\label{eq:injective_necessary_condition}
{A_q(u,\ell,2(\ell-t))} \leq \qbin{n-u}{k-\ell}{q}.
\end{equation}
\end{corollary}

\begin{proof}
Suppose there exist two distinct spaces $A_1,A_2$ with the same image under $\varphi$.
By \cref{prop:phi_injective} and the assumption $k>\ell+t$, we get 
$$ \dim(A_1\cap A_2) \leq {\ell+t-k}<0, $$
which is impossible as the dimension of the intersection of two subspaces.
Hence, in this case no two distinct elements of $\cA$ can have the same image under $\varphi$, i.e. $\varphi$ is injective, implying that $|\cA|\leq |\cB|$.
Since
$$
|\cA|={A_q(u,\ell,2(\ell-t))}
\qquad\text{and}\qquad
|\cB|=\qbin{n-u}{k-\ell}{q},
$$
the claimed inequality follows.
\end{proof}

By rephrasing \cref{prop:phi_injective} in terms of powers of the Johnson graph, we get the following.

\begin{proposition}\label{cor:indep_johnson}
For every $B\in\cB$, the set $\varphi^{-1}(B)$ is an independent set of
$J_q(u,\ell)^{k-{t-1}}$.
\end{proposition}

\begin{proof}
Let $B\in\cB$ and consider two distinct $A_1,A_2\in \varphi^{-1}(B)$. By
\cref{prop:phi_injective}, we have
$$
\dim(A_1\cap A_2)\leq {\ell+t-k}.
$$
On the other hand, two vertices $A_1,A_2$ of $J_q(u,\ell)^{k-{t-1}}$ are adjacent precisely
when
$$
\dim(A_1\cap A_2)\geq \ell-(k-{t-1})>{\ell+t-k}.
$$
Thus, no two distinct elements of $\varphi^{-1}(B)$ are adjacent in
$J_q(u,\ell)^{k-{t-1}}$. Hence, $\varphi^{-1}(B)$ is an independent set.
\end{proof}

Notice that if $J_q(u,\ell)^{k-{t-1}}$ is the complete graph, then every independent set has size at most one. Hence, in this case, every nonempty preimage $\varphi^{-1}(B)$ contains exactly one
element, and $\varphi$ is injective.

More generally, $\varphi{^{-1}}$ partitions the vertices of $J_q(u,\ell)^{k-{t-1}}$ into independent sets. Therefore, we have the following.

\begin{corollary}\label{cor:chromatic_necessary_condition}
{If the Schubert subspace code $\mS_\varphi$ is an $(\ell,t)$-intersecting set, then
$$
\chi(J_q(u,\ell)^{k-t-1}[\cA])\leq \qbin{n-u}{k-\ell}{q}.
$$}
\end{corollary}

\begin{proof}
By \cref{cor:indep_johnson}, for every $B\in\cB$, the preimage $\varphi^{-1}(B)$ is an
independent set of $J_q(u,\ell)^{k-t-1}$ and hence of $J_q(u,\ell)^{k-t-1}[\cA]$.
Thus, the nonempty preimages of elements of $\cB$ define a proper coloring of the graph $J_q(u,\ell)^{k-t-1}[\cA]$.
The number of such preimages is at most $|\cB|$, and therefore
$$
\chi\big(J_q(u,\ell)^{k-t-1}[\cA]\big)
\leq
|\cB|
=
\qbin{n-u}{k-\ell}{q}.
$$
\end{proof}

This condition is the natural converse to the coloring method used in
\cref{prop:chi_small_is_sufficient}.
There, one starts with a coloring of $J_q(u,\ell)^{k-t-1}$ and assigns a disjoint element of a partial $(k-\ell)$-spread in $V$ to each color.
{ As pointed out in \cref{rem:strengthening_chi_subgraph}, it is sufficient that the partial spread is larger than the chromatic number of the subgraph of $J_q(u,\ell)^{k-t-1}$ induced by $\cA$ instead of the whole graph.}
Conversely, any construction of the form $\mS_\varphi$ necessarily induces a coloring of $J_q(u,\ell)^{k-{t-1}}{[\cA]}$ by grouping together all elements of $\cA$ with the same image under $\varphi$. In the special case $t=\ell-1$, we have that $J_q(u,\ell)^{k-t-1}[\cA]=J_q(u,\ell)^{k-\ell}$ so the condition in \cref{cor:chromatic_necessary_condition} can be asymptotically studied, leading to the following result.

\begin{corollary}\label{cor:necess_t=l-1}
Assume that {$t=\ell-1$} and $k-\ell<\min\{\ell,u-\ell\}$. If $\mS_\varphi$ is an
$(\ell,\ell-1)$-intersecting set, then, for $q\to\infty$, we must have  
$$
\begin{cases}
u \leq n-k & \textnormal{if } u\leq 2\ell,\\[1mm]
u \leq \dfrac{2\ell+n-k}{2} & \textnormal{if } u>2\ell.
\end{cases}
$$
\end{corollary}

\begin{proof}
By \cref{cor:chromatic_necessary_condition}, we have
$$
\chi\big(J_q(u,\ell)^{k-\ell}\big)\leq \qbin{n-u}{k-\ell}{q}.
$$
Since $k-\ell<\min\{\ell,u-\ell\}$, the graph $J_q(u,\ell)^{k-\ell}$ is not complete, and we can use the asymptotic estimate \eqref{eq:asymptotic_chromatic} from \cite{dhaeseleer2026chromatic} for the chromatic number of powers of $q$-Johnson graphs.
Thus, for fixed $n,k,u,\ell,t$ and $q\to\infty$, the left-hand side has order
$$
q^{(k-\ell)\max\{u-\ell,\ell\}}.
$$
On the other hand, the right-hand side has order
$$
q^{(k-\ell)(n-u-(k-\ell))}.
$$
Therefore, we must have
$$
(k-\ell)\max\{u-\ell,\ell\}
\leq
(k-\ell)(n-u-(k-\ell)).
$$

Since $k-\ell>0$, this is equivalent to
$$
\max\{u-\ell,\ell\}\leq n-u-k+\ell.
$$
If $u\leq 2\ell$, then $\max\{u-\ell,\ell\}=\ell$. Hence,
$\ell\leq n-u-k+\ell$, and so $u\leq n-k$.\\
If $u>2\ell$, then
$\max\{u-\ell,\ell\}=u-\ell$. Hence, $u-\ell\leq n-u-k+\ell$, and so
$2u\leq n-k+2\ell$, that is,
$$
u\leq \frac{2\ell+n-k}{2}.
$$
\end{proof}

\begin{remark}{If $k-t-1\geq \min\{\ell,u-\ell\}$, then the graph $J_q(u,\ell)^{k-t-1}$ is complete and hence also $J_q(u,\ell)^{k-t-1}[\cA]$ is complete. In
this case
$$
\chi(J_q(u,\ell)^{k-t-1}[\cA])=A_q(u,\ell,2(\ell-t)),
$$
and the necessary condition from \cref{cor:chromatic_necessary_condition} becomes
$$
A_q(u,\ell,2(\ell-t))\leq \qbin{n-u}{k-\ell}{q}.
$$
This is consistent with the injectivity of $\varphi$: in the complete case, no two distinct
elements of $\cA$ can have the same image under $\varphi$.}
\end{remark}

We can also obtain necessary conditions from \emph{cliques}. Recall that a \textbf{clique} in a graph is a set of vertices which are pairwise adjacent. We denote by $\omega(G)$ the clique number of a graph $G$, namely the maximum size of a clique in $G$.

In the following result, we prove that any clique in {the subgraph of} $J_q(u,\ell)^{\ell-t}$ {induced by $\cA$} corresponds, via $\varphi$, to a partial $(k-\ell)$-spread of $V$.
\begin{proposition}\label{prop:clique<spread}
    Let $\mS_\varphi$ be a Schubert subspace code that is an $(\ell,t)$-intersecting set. Then, there exists a partial $(k-\ell)$-spread $\cP$ of $V$ such that
    $$\omega(J_q(u,\ell)^{\ell-t}[\cA])= |\cP|.$$
\end{proposition}
\begin{proof}
Let $\mathcal{K}$ be a clique of maximal size in $J_q(u,\ell)^{\ell-t}[\cA]$. Thus, $|\mathcal{K}|=\omega\big(J_q(u,\ell)^{\ell-t}[\cA]\big)$.
For every $A\in \mathcal{K}$, the subspace $\varphi(A)$ is contained in $V$ and has dimension $k-\ell$.
If $A_1, A_2 \in \cK$ are distinct, then they are adjacent in $J_q(u,\ell)^{\ell-t}$, and hence
$\dim(A_1\cap A_2)\ge t$. Since $\mS_\varphi$ is an $(\ell,t)$-intersecting set with $t\leq\ell-1$, \cref{lem:dim(S1capS2)} gives
\begin{align*}
\dim\big(\varphi(A_1)\cap\varphi(A_2)\big)
&=
\dim\big((A_1\oplus\varphi(A_1))\cap(A_2\oplus\varphi(A_2))\big) - \dim(A_1\cap A_2) \\
&\leq t- t = 0.
\end{align*}
Hence, $\varphi(A_1)\cap\varphi(A_2)=\{0\}$ and 
$\cP:=\{\varphi(A) : A \in\cK\}$ is a partial $(k-\ell)$-spread of $V$. Therefore,
$$
\omega\big(J_q(u,\ell)^{\ell-t}[\cA]\big)=|\cK|=|\cP|,
$$
which proves the claim.
\end{proof}

In general, determining the cliques of the subgraph $J_q(u,\ell)^{\ell-t}[\cA]$ seems a difficult problem, since it depends on the structure and the distance distribution of the subspace code $\cA$. However, in the special case $t=\ell-1$, the subspace code $\cA$ is the whole Grassmannian $\Gr_q(\ell, u)$, and the graph we are considering is $J_q(u,\ell)$ without taking its powers. 
The maximal cliques in $J_q(u,\ell)$ are known, see e.g.~\cite[Sec.~9.3]{brouwer1989distance}.
They are of one of the following two types:
\begin{enumerate}[label=(\roman*)]
    \item collections of all $\ell$-dimensional subspaces of $U$ containing a
fixed $(\ell-1)$-subspace of $U$, also known as \textit{sunflowers}. These cliques have size $$
\qbin{u-\ell+1}{1}{q};
$$
    \item  collections of all
$\ell$-dimensional subspaces of $U$ contained in a fixed $(\ell+1)$-subspace of $U$.
These cliques have size
$$
\qbin{\ell+1}{1}{q}.
$$
\end{enumerate}
By using these two families, we obtain the following asymptotic estimate.

\begin{proposition}\label{prop:clique_upper_cond}
If $t=\ell-1$ and the Schubert subspace code $\mS_\varphi$ is an $(\ell,\ell-1)$-intersecting set, then, for $q\to\infty$, it must hold that
$$
u\leq \min\left\{n-k,\frac{2\ell+n-k}{2}\right\}.
$$
\end{proposition}

\begin{proof}
By the discussion above, since $t=\ell-1$, the induced subgraph $J_q(u,\ell)[\cA]$ is the whole $q$-Johnson graph. Hence, by \cref{prop:clique<spread}, any clique in $J_q(u,\ell)$ gives rise, via $\varphi$, to a partial $(k-\ell)$-spread in $V$.
In particular, the cardinality of any clique in $J_q(u,\ell)$ must be at most the cardinality of a partial $(k-\ell)$-spread in $V$.

The two families of maximal cliques described above have cardinalities respectively
$$
\qbin{u-\ell+1}{1}{q},
\qquad\textnormal{and}\qquad
\qbin{\ell+1}{1}{q}.
$$
On the other hand, since each nonzero vector can belong to at most one space, a partial $(k-\ell)$-spread in $V$ has cardinality at most
$$
\frac{q^{n-u}-1}{q^{k-\ell}-1}.
$$
Therefore, we must have
$$
\qbin{u-\ell+1}{1}{q}
\leq
\frac{q^{n-u}-1}{q^{k-\ell}-1}
\quad \text{and}\quad
\qbin{\ell+1}{1}{q}
\leq
\frac{q^{n-u}-1}{q^{k-\ell}-1}.
$$
For $q\to\infty$, these two inequalities imply respectively
$$
u-\ell \leq n-u-k+\ell
\quad \text{and}\quad
\ell\leq n-u-k+\ell.
$$
Equivalently, we get
$$
u\leq \frac{2\ell+n-k}{2}
\quad \text{and}\quad
u\leq n-k.
$$
Hence,
$$
u\leq \min\left\{n-k,\frac{2\ell+n-k}{2}\right\}.
$$
\end{proof}

\begin{remark}
The chromatic condition from \cref{cor:chromatic_necessary_condition}
$$
\chi\big(J_q(u,\ell)^{k-\ell}\big)\leq \qbin{n-u}{k-\ell}{q}
$$
and the condition from \cref{prop:clique_upper_cond} use different consequences of the
assumption that $\mS_\varphi$ is an $(\ell,\ell-1)$-intersecting set.
The first condition only uses the fibers of $\varphi$. Indeed, each fiber must be an independent
set in $J_q(u,\ell)^{k-\ell}$, and therefore the fibers of $\varphi$ induce a coloring of this
graph with at most $\qbin{n-u}{k-\ell}{q}$ colors.
The second condition uses the stronger fact that adjacent vertices of $J_q(u,\ell)$ must be
mapped to pairwise disjoint elements of $\cB$. Thus, every clique in $J_q(u,\ell)$ gives rise,
via $\varphi$, to a partial $(k-\ell)$-spread inside $V$.
In the non-complete range $k-\ell<\min\{\ell,u-\ell\}$, the first condition gives the bound in \cref{cor:necess_t=l-1}. The second condition gives the single bound
$$
u\leq \min\left\{n-k,\frac{2\ell+n-k}{2}\right\},
$$
which implies both pieces of \cref{cor:necess_t=l-1}. In this sense, the condition from
\cref{prop:clique_upper_cond} is stronger, although it follows from a different consequence of
the same intersection requirement.
\end{remark}

\subsection{Comparison between necessary and sufficient conditions}
We now compare the sufficient conditions obtained in Section~\ref{sec:sufficient_conditions} with the necessary conditions obtained in Section~\ref{sec:necessary_conditions}.

The sufficient condition in \cref{prop:chi_small_is_sufficient} requires a partial $(k-\ell)$-spread in $V$ whose cardinality is at least $\chi\big(J_q(u,\ell)^{k-t-1}\big)$.
Using Beutelspacher's construction of partial spreads of \cref{thm:beutelspacher_partial_spread} together with the asymptotic estimate for the chromatic number of powers of Grassmann graphs from \cite{dhaeseleer2026chromatic}, \cref{cor:asymptotic_general_ell} gives, in the non-complete range of the power graph, the sufficient condition
$$
(k-t-1)\max\{u-\ell,\ell\}< n-u-(k-\ell).
$$

In order to explicitly compare this with the necessary conditions, we assume that $t=\ell-1$, so that $\cA=\Gr_q(\ell,u)$ and the relevant power of the $q$-Johnson graph is $J_q(u,\ell)^{k-\ell}$. In this case, assume first that we are in the non-complete range
$$
k-\ell<\min\{\ell,u-\ell\}.
$$
Then the sufficient condition above becomes
$$
(k-\ell)\max\{u-\ell,\ell\}< n-u-(k-\ell).
$$
On the other hand, \cref{cor:necess_t=l-1} gives, again in the non-complete range, the necessary condition
$$
\max\{u-\ell,\ell\}\leq n-u-(k-\ell).
$$
Thus, there is a factor $k-\ell$ on the left-hand side of the sufficient condition. This gap
comes from the fact that, in \cref{prop:chi_small_is_sufficient}, different color classes are assigned to pairwise disjoint $(k-\ell)$-dimensional subspaces of $V$. Hence, the number of
colors must be bounded by the size of a partial $(k-\ell)$-spread in $V$, whose asymptotic order
is $q^{n-u-(k-\ell)}$. The necessary condition from
\cref{cor:chromatic_necessary_condition}, instead, only uses the fact that $\varphi$ takes
values in the whole set $\cB$, whose asymptotic order is
$q^{(k-\ell)(n-u-(k-\ell))}$. This explains the factor $k-\ell$ appearing in the comparison.

In the complete range, the coloring-based argument gives no improvement over the injective
assignment in \cref{prop:partial_spread_simple}. Indeed, if $J_q(u,\ell)^{k-\ell}$ is complete,
then
$$
\chi\big(J_q(u,\ell)^{k-\ell}\big)=\qbin{u}{\ell}{q},
$$
and \cref{prop:chi_small_is_sufficient} reduces to asking for a partial
$(k-\ell)$-spread in $V$ of size at least $\qbin{u}{\ell}{q}$.

The necessary condition obtained from cliques in \cref{prop:clique_upper_cond} is
closer in spirit to the condition used in \cref{prop:chi_small_is_sufficient}. Indeed,
it uses the fact that adjacent vertices of $J_q(u,\ell)$ must be mapped by $\varphi$ to
pairwise disjoint elements of $\cB$. Therefore, large cliques in $J_q(u,\ell)$ force the
existence of large partial $(k-\ell)$-spreads inside $V$. This yields the necessary condition
$$
u\leq \min\left\{n-k,\frac{2\ell+n-k}{2}\right\}.
$$
In general, this condition is independent of the condition coming from the induced coloring:
the latter counts how many different values of $\varphi$ are available, while the condition
obtained from cliques detects when some of these values must be mutually disjoint.

The following example illustrates the gap between the sufficient and necessary conditions.

\begin{example}
Let $u=6$, $\ell=3$ and $k=5$. Then $k-\ell=2$, and the graph appearing in
\cref{prop:chi_small_is_sufficient} is $J_q(6,3)^2$. Equivalently, this is the graph
whose vertices are the $3$-dimensional subspaces of $U$, with two vertices adjacent whenever
they intersect in dimension at least $1$.

The construction of \cite{dhaeseleer2026chromatic} gives a proper coloring of $J_q(6,3)^2$ with $10q^6$ colors.
Hence, by the sufficient condition of \cref{prop:chi_small_is_sufficient}, it is enough to find a partial $2$-spread in $V$ of size at least $10q^6$.
By Beutelspacher's construction of \cref{thm:beutelspacher_partial_spread}, there exists a partial $2$-spread of size
$$
\frac{q^{n-6}-q^{2+r}}{q^2-1}+1,
$$
where $r$ is the remainder of $n-6$ modulo $2$.
Therefore, the sufficient condition is
$$
\frac{q^{n-6}-q^{2+r}}{q^2-1}+1\geq 10q^6.
$$
Direct computations show that, for even $n$, it holds for all $q\geq 2$ and all even $n\geq 16$, except for $(q,n)=(2,16)$. Moreover, for odd $n$, it holds for all
$q\geq 2$ and all odd $n\geq 17$, and  for $n=15$ and $q\geq 11$.

The necessary condition from \cref{cor:chromatic_necessary_condition} is
$$
\chi\big(J_q(6,3)^2\big)\leq \qbin{n-6}{2}{q}.
$$
Since $\chi\big(J_q(6,3)^2\big)$ has order $q^6$, while
$\qbin{n-6}{2}{q}$ has order $q^{2(n-8)}$, this condition is satisfied, asymptotically in $q$, whenever $6< 2(n-8)$. 
That is, as stated also in \cref{cor:necess_t=l-1}, whenever $n> 11$.

This example shows that the gap between the necessary and sufficient conditions comes from the
difference between the number of available elements of $\cB$ and the number of pairwise
disjoint elements of $\cB$. Indeed,
$$
|\cB|=\qbin{n-6}{2}{q}
$$
has order $q^{2(n-8)}$, whereas the partial $2$-spread in $V$ from \cref{thm:beutelspacher_partial_spread} has order only $q^{n-8}$.
The condition from \cref{cor:necess_t=l-1} compares the chromatic number, of order
$q^6$, with $q^{2(n-8)}$, while \cref{prop:chi_small_is_sufficient} compares the same chromatic number with $q^{n-8}$.
\end{example}

\section{Constructions from evasive and scattered subspaces}\label{sec:scattered_generalized}

In this section, we describe a second family of constructions of Schubert subspace codes,
obtained via field reduction from subspaces over an extension field.
The construction is naturally phrased in terms of evasive subspaces, while $h$-scattered subspaces provide a particularly clean case in which the cardinality can be computed exactly.
For $h=1$, this recovers the construction from scattered $\Fq$-subspaces given in \cite{alfarano2024schubert} for the case $\ell=1$ and $t=0$.

Let $r,k$ be positive integers.
We identify $\F_{q^k}^r$ with $\Fq^{rk}$ via the \textit{field-reduction map} $\psi:\F_{q^k}^r\longrightarrow \Fq^{rk}$, which expands coordinates with respect to a basis of $\bF_{q^k}$ over $\bF_q$.
Thus, every $s$-dimensional $\F_{q^k}$-subspace of $\F_{q^k}^r$ gives rise to an $sk$-dimensional $\Fq$-subspace of $\Fq^{rk}$.

Throughout this section, let $\cU$ be a $u$-dimensional $\Fq$-subspace of $\F_{q^k}^r$ and denote by $$U:=\psi(\cU)$$ the \textit{field-reduced} space of $\cU$.
We first recall the definition of \textit{evasive} and \textit{$h$-scattered} spaces; see \cite{blokhuis2000scattered,bartoli2021evasive,csajbok2021generalising}.

\begin{definition}
    Let $a<r$ and let $b$ be a positive integer.
    The subspace $\cU$ is called \textbf{$(a,b)$-evasive} if $\dim_{\F_{q^k}}\langle \cU\rangle_{\F_{q^k}}=r$ and, for every $a$-dimensional $\F_{q^k}$-subspace $\cW$ of $\F_{q^k}^r$, one has $\dim_{\Fq}(\cU\cap \cW)\leq b$.
    For $1\leq h<r$, an $(h,h)$-evasive space $\cU$ is called \textbf{$h$-scattered}.
\end{definition}

For $1\leq \ell<r$, consider
$$
\Gr_{q^k}(\ell,r):=
\{A\subseteq_{\F_{q^k}}\F_{q^k}^r:\dim_{\F_{q^k}}(A)=\ell\}.
$$
We define
$$
\Lambda_{\ell,\cU}:=
\{A\in\Gr_{q^k}(\ell,r):\dim_{\Fq}(A\cap\cU)=\ell\},
\qquad
\mS_{\ell,U}:=\{\psi(A):A\in\Lambda_{\ell,\cU}\}.
$$
The elements of $\Lambda_{\ell,\cU}$ are the $\ell$-dimensional $\F_{q^k}$-subspaces whose intersection with $\cU$ has dimension exactly $\ell$ over $\Fq$.
If $\cU$ is $\ell$-scattered, then the condition $\dim_{\Fq}(A\cap\cU)=\ell$ may equivalently be replaced by $\dim_{\Fq}(A\cap\cU)\geq \ell$.

\begin{proposition}\label{prop:field_reduction_general}
The set $\mS_{\ell,U}$ is an $(\ell,(\ell-1)k)$-intersecting set with respect to $U$ in
$\Gr_q(\ell k,rk)$. Moreover, $|\mS_{\ell,U}|=|\Lambda_{\ell,\cU}|$.
\end{proposition}

\begin{proof}
For every $A\in\Lambda_{\ell,\cU}$, the field-reduced space $\psi(A)$ has dimension
$\ell k$ over $\Fq$ and belongs to the Schubert variety $\Omega_{U,\ell}$. If
$A,B\in\Lambda_{\ell,\cU}$ are distinct, then
$\dim_{\F_{q^k}}(A\cap B)\leq \ell-1$, and hence
$$
\dim_{\Fq}(\psi(A)\cap\psi(B))
=
k\cdot\dim_{\F_{q^k}}(A\cap B)
\leq (\ell-1)k.
$$
Thus, $\mS_{\ell,U}$ has the claimed intersection property. Finally, field reduction is injective on $\F_{q^k}$-subspaces, so $|\mS_{\ell,U}|=|\Lambda_{\ell,\cU}|$.
\end{proof}

We now specialize to $h$-scattered subspaces. In this case, the size of the construction can be
computed exactly. The key point is the following.

\begin{lemma}\label{lem:h_subspace_spans_h}
Let $\cU$ be an $h$-scattered $\Fq$-subspace of $\F_{q^k}^r$, and let $X\subseteq \cU$ be such
that $\dim_{\Fq}(X)=h$. Then
$\dim_{\F_{q^k}}(\langle X\rangle_{\F_{q^k}})=h$.
\end{lemma}

\begin{proof}
Suppose by contradiction that $\dim_{\F_{q^k}}(\langle X\rangle_{\F_{q^k}})\leq h-1$, and set
$T:=\langle X\rangle_{\F_{q^k}}$. Since $\langle \cU\rangle_{\F_{q^k}}=\F_{q^k}^r$ and
$\dim_{\F_{q^k}}(T)<r$, there exists $v\in\cU\setminus T$.

The space $T+\langle v\rangle_{\F_{q^k}}$ has $\F_{q^k}$-dimension at most $h$. If its
dimension is smaller than $h$, enlarge it to an $h$-dimensional $\F_{q^k}$-subspace $H$ of
$\F_{q^k}^r$; otherwise, set $H:=T+\langle v\rangle_{\F_{q^k}}$. Then $H$ contains both $X$
and $v$. Since $v\notin T$ and $X\subseteq T$, we have
$\dim_{\Fq}(X+\langle v\rangle_{\Fq})=h+1$. Moreover,
$X+\langle v\rangle_{\Fq}\subseteq \cU\cap H$, contradicting the fact that $\cU$ is
$h$-scattered.
\end{proof}

\begin{theorem}\label{thm:h_scattered_size}
Let $\cU$ be an $h$-scattered $\Fq$-subspace of $\F_{q^k}^r$ of dimension $u$ and let $U$ be its field-reduced space.
Then $\mS_{h,U}$ is an $(h,(h-1)k)$-intersecting set with respect to $U$ in $\Gr_q(hk,rk)$, and
$$
|\mS_{h,U}|=\qbin{u}{h}{q}.
$$
\end{theorem}

\begin{proof}
The intersection property follows from \cref{prop:field_reduction_general}. It remains to
compute the cardinality. We show that
$A\mapsto A\cap\cU$ gives a bijection from $\Lambda_{h,\cU}$ to the family of all $h$-dimensional $\bF_q$-subspaces of $\cU$.

If $A\in\Lambda_{h,\cU}$, then $A\cap\cU$ is an $h$-dimensional $\Fq$-subspace of $\cU$.
Conversely, let $X\subseteq \cU$ be an $h$-dimensional $\Fq$-subspace.
By \cref{lem:h_subspace_spans_h}, the space $A_X:=\langle X\rangle_{\F_{q^k}}$ has $\F_{q^k}$-dimension $h$.
Since $X\subseteq A_X\cap\cU$ and $\cU$ is $h$-scattered, we have $A_X\cap\cU=X$, and hence $A_X\in\Lambda_{h,\cU}$.
Finally, if $A_X\in\Lambda_{h,\cU}$ and $X=A_X\cap\cU$, then $\langle X\rangle_{\F_{q^k}}\subseteq A_X$.
By \cref{lem:h_subspace_spans_h}, both spaces have $\F_{q^k}$-dimension $h$, so $A_X=\langle X\rangle_{\F_{q^k}}$.
Therefore, the map is bijective and $|\Lambda_{h,\cU}|=\qbin{u}{h}{q}$.

Since $|\mS_{h,U}|=|\Lambda_{h,\cU}|$, the result follows.
\end{proof}

\begin{remark}
For $h=1$, the theorem gives a $(1,0)$-intersecting set in $\Gr_q(k,rk)$ of size
$\qbin{u}{1}{q}$. This recovers the construction from scattered $\Fq$-subspaces given in
\cite{alfarano2024schubert}. Indeed, in this case, the codewords are the field reductions of
the $1$-dimensional $\F_{q^k}$-subspaces meeting $\cU$ in dimension $1$, and the scatteredness
of $\cU$ guarantees that any two distinct codewords intersect trivially.
\end{remark}

\begin{proposition}\label{prop:field_reduction_upper_bound}
Let $\mS\subseteq \Omega_{U,h}\cap \Gr_q(hk,rk)$ be an $(h,(h-1)k)$-intersecting set with respect to $U$. Then
$$
|\mS|
\leq
\qbin{u}{h}{q}
\frac{
\qbin{rk-h}{(h-1)(k-1)}{q}
}{
\qbin{hk-h}{(h-1)(k-1)}{q}
}.
$$
\end{proposition}

\begin{proof}
For every $h$-dimensional subspace $X\subseteq U$, let $\mS_X:=\{S\in\mS:X\subseteq S\}$.
Since every $S\in\mS$ contains at least one $h$-dimensional subspace of $U$, we have
$$
|\mS|\leq \sum_{X\subseteq U,\ \dim(X)=h}|\mS_X|.
$$
Fix such an $X$. For every $S\in\mS_X$, the quotient $S/X$ has dimension $hk-h$ inside
$\Fq^{rk}/X$. If $S_1,S_2\in\mS_X$ are distinct, then
$$
\dim((S_1/X)\cap(S_2/X))
=
\dim(S_1\cap S_2)-h
\leq
(h-1)k-h.
$$
Therefore, no two distinct quotients $S/X$ can contain the same subspace of dimension
$(h-1)k-h+1=(h-1)(k-1)$. Counting such subspaces in $\Fq^{rk}/X$, we obtain
$$
|\mS_X|
\qbin{hk-h}{(h-1)(k-1)}{q}
\leq
\qbin{rk-h}{(h-1)(k-1)}{q}.
$$
Thus, we have
$$
|\mS_X|
\leq
\frac{
\qbin{rk-h}{(h-1)(k-1)}{q}
}{
\qbin{hk-h}{(h-1)(k-1)}{q}
}.
$$
Since there are $\qbin{u}{h}{q}$ choices for $X\subseteq U$, the result follows.
\end{proof}

\begin{remark}
For $h=1$, the extra factor in \cref{prop:field_reduction_upper_bound} is equal to $1$, and the bound becomes $|\mS|\leq \qbin{u}{1}{q}$, which matches the size of the code obtained from $1$-scattered subspaces.
For $h>1$, the bound is generally larger than $\qbin{u}{h}{q}$, and therefore it does not prove optimality of the code $\mS_{h,U}$ obtained from an $h$-scattered subspace.
It does, however, give an explicit upper bound for all Schubert subspace codes with parameters $\ell=h$, $t=(h-1)k$, and codeword dimension $hk$.
\end{remark}

\section{Conclusion and open problems}\label{sec:conclusions}

We constructed new families of Schubert subspace codes attaining the bound in \cref{prop:upper_bound} in several extremal cases.
The direct-sum construction relates optimality to partial spreads and colorings of powers of Grassmann graphs, while the field-reduction construction connects the problem with evasive and scattered subspaces.

Several questions remain open.

\begin{enumerate}
    \item Close the gap between the sufficient conditions from Section~\ref{sec:sufficient_conditions} and the necessary conditions given on the chromatic and clique numbers in Section~\ref{sec:necessary_conditions}.
    
    \item Study whether there exist optimal constructions that are not of direct-sum type, and whether such constructions do not rely on partial spreads.
    
    \item For the field-reduction construction with $h>1$, determine whether the codes obtained from
    $h$-scattered subspaces are optimal among all Schubert subspace codes with the same parameters.
    
    \item Investigate sharp upper bounds for Schubert subspace codes in the regime where the counting bound of \cref{prop:upper_bound} does not apply.
\end{enumerate}

\bigskip
	
\bibliographystyle{abbrv}
\bibliography{bibliography.bib}

@article{alfarano2024schubert,
  title={Schubert subspace codes},
  author={Alfarano, Gianira N and Rosenthal, Joachim and Toesca, Beatrice},
  journal={J. Algebra Appl.},
  volume={24},
  number={13n14},
  pages={2541006},
  year={2025},
  publisher={World Scientific}
}

@article{csajbok2021generalising,
  title={Generalising the scattered property of subspaces},
  author={Csajb{\'o}k, Bence and Marino, Giuseppe and Polverino, Olga and Zullo, Ferdinando},
  journal={Combinatorica},
  volume={41},
  number={2},
  pages={237--262},
  year={2021},
  publisher={Springer}
}

@article{dhaeseleer2025chromatic,
  title={On the Chromatic Number of {G}rassmann Graphs},
  author={D'haeseleer, Jozefien and Taranchuk, Vladislav},
  journal={Linear Algebra Its Appl.},
  year={2026},
  publisher={Elsevier}
}

@article{dhaeseleer2026chromatic,
  title={Chromatic Number of {G}rassmann Graphs and {MRD} codes},
  author={D'haeseleer, Jozefien and Pavese, Francesco and Santonastaso, Paolo and Taranchuk, Vladislav},
  journal={arXiv preprint arXiv:2602.10777},
  year={2026}
}

@INPROCEEDINGS{ma08p,
        AUTHOR = {Manganiello, F. and Gorla, E. and Rosenthal, J.},
        TITLE  = {Spread Codes and Spread Decoding in Network Coding},
        ADDRESS            = {Toronto, Canada},
        BOOKTITLE          = {IEEE Int. Symp. Inf. Theory - Proc. 2008},
        PAGES              = {851--855},
        YEAR               = 2008,
        doi={10.1109/ISIT.2008.4595113},
}

@article{go12,
  author    = {E. Gorla and F. Manganiello and J. Rosenthal},
  title     = {An Algebraic Approach for Decoding Spread Codes},
  journal   = {Adv. in Math. of Commun.}, 
  volume =	 6,
  number =	 4,
  year      = 2012,
  pages =	 {443--466},
  ee        = {http://arxiv.org/abs/1107.5523},
  bibsource = {DBLP, http://dblp.uni-trier.de}
}

@article{beutelspacher1975partial,
  title={Partial spreads in finite projective spaces and partial designs},
  author={Beutelspacher, Albrecht},
  journal={Math. Z.},
  volume={145},
  number={3},
  pages={211--229},
  year={1975},
  publisher={Springer}
}

@book{brouwer1989distance,
  title={Distance-regular graphs},
  author={Brouwer, Andries E and Cohen, Arjeh M. and Neumaier, A.},
  year={1989},
  publisher={Springer}
}

@article{brooks1941colouring,
  title={On colouring the nodes of a network},
  author={Brooks, Rowland Leonard},
  journal={Math. Proc. Camb. Philos. Soc.},
  volume={37},
  number={2},
  pages={194--197},
  year={1941},
  organization={Cambridge University Press}
}

@article{koetter2008coding,
  title={Coding for errors and erasures in random network coding},
  author={Koetter, Ralf and Kschischang, Frank R},
  journal={IEEE Trans. Inf. Theory},
  volume={54},
  number={8},
  pages={3579--3591},
  year={2008},
  publisher={IEEE}
}

@article{kurz2021constructions,
  title={Constructions and bounds for subspace codes},
  author={Kurz, Sascha},
  journal={arXiv preprint arXiv:2112.11766},
  year={2021}
}

@article{horlemann2018constructions,
  title={Constructions of constant dimension codes},
  author={Horlemann-Trautmann, Anna-Lena and Rosenthal, Joachim},
  journal={Network Coding and Subspace Designs},
  pages={25--42},
  year={2018},
  publisher={Springer}
}

@article{etzion2009error,
  title={Error-correcting codes in projective spaces via rank-metric codes and {F}errers diagrams},
  author={Etzion, Tuvi and Silberstein, Natalia},
  journal={IEEE Trans. Inf. Theory},
  volume={55},
  number={7},
  pages={2909--2919},
  year={2009},
  publisher={IEEE}
}

@article{segre1964teoria,
  title={Teoria di {G}alois, fibrazioni proiettive e geometrie non desarguesiane},
  author={Segre, Beniamino},
  journal={Ann. Mat. Pura Appl.},
  volume={64},
  number={1},
  pages={1--76},
  year={1964},
  publisher={Springer}
}

@article{heinlein2017coset,
  author  = {Daniel Heinlein and Sascha Kurz},
  title   = {Coset Construction for Subspace Codes},
  journal = {IEEE Trans. Inf. Theory},
  volume  = {63},
  number  = {12},
  pages   = {7651--7660},
  year    = {2017},
  doi     = {10.1109/TIT.2017.2753822}
}

@article{heinlein2016tables,
  title={Tables of subspace codes},
  author={Heinlein, Daniel and Kiermaier, Michael and Kurz, Sascha and Wassermann, Alfred},
  journal={arXiv preprint arXiv:1601.02864},
  year={2016}
}

@book{fulton1997young,
	title={Young tableaux: with applications to representation theory and geometry},
	author={Fulton, William},
	year={1997},
        series={London Mathematical Society Student Texts},
        number={35},
	publisher={Cambridge University Press}
}

@book{hodge1947methods,
	title={Methods of algebraic geometry},
	author={Hodge, William Vallance Douglas and Pedoe, Daniel},
	volume={2},
	year={1947},
	publisher={Cambridge University Press}
}

@inproceedings{ghorpade2000higher,
  title={Higher weights of {G}rassmann codes},
  author={Ghorpade, Sudhir R and Lachaud, Gilles},
  booktitle={Coding Theory, Cryptography and Related Areas: Proceedings of an International Conference on Coding Theory, Cryptography and Related Areas, held in Guanajuato, Mexico, in April 1998},
  pages={122--131},
  year={2000},
  organization={Springer}
}

@article{bartoli2021evasive,
  title={Evasive subspaces},
  author={Bartoli, Daniele and Csajb{\'o}k, Bence and Marino, Giuseppe and Trombetti, Rocco},
  journal={J. Comb. Des.},
  volume={29},
  number={8},
  pages={533--551},
  year={2021},
  publisher={Wiley Online Library}
}

@article{blokhuis2000scattered,
  title={Scattered spaces with respect to a spread in {PG}$(n,q)$},
  author={Blokhuis, Aart and Lavrauw, Michel},
  journal={Geom. Dedicata},
  volume={81},
  number={1},
  pages={231--243},
  year={2000},
  publisher={Springer}
}

\end{document}